\documentclass[reqno, 11pt]{amsart}

 \usepackage{amssymb}
 \usepackage{amsthm}
 \usepackage{amsmath}
 \usepackage{times}
 \usepackage{latexsym}
 \usepackage{xcolor}
 \usepackage{soul}
 \usepackage[mathscr]{eucal}

 \numberwithin{equation}{section}
 
 \newtheorem{theorem}{Theorem}[section]
 \newtheorem{proposition}[theorem]{Proposition}
 \newtheorem{lemma}[theorem]{Lemma}
 \newtheorem{corollary}[theorem]{Corollary}

 \newtheorem{example}[theorem]{Example}

 \title[Locally conformal almost generalized $f$-cosymplectic manifolds]{Locally conformal almost generalized $f$-cosymplectic manifolds}
 \author[Fortun\'e Massamba, Jude Rosnick Bayeni Mitoueni]{Fortun\'e Massamba*, Jude Rosnick Bayeni Mitoueni** }
 \newcommand{\acr}{\newline\indent}
 \address{\llap{*\,}Discipline of Mathematics\acr 
 	School of Agriculture and Science\acr
 	University of KwaZulu-Natal\acr
 	Private Bag X01, Scottsville 3209\acr 
 	South Africa} 
 \email{massfort@yahoo.fr, Massamba@ukzn.ac.za} 
 \thanks{}
 \address{\llap{**\,}Universite Marien Ngouabi\acr
 	Facult\'e des Sciences et Technique Parcours Mathematiques\acr
 	BP. 69, Brazzaville, Congo}  
 \email{judebayeni@gmail.com} 
 \thanks{}  
 
 \dedicatory{}
 
\begin{document}

\begin{abstract}
This paper introduces a new class of geometric structures in almost contact metric geometry, which we call locally conformal almost generalized $f$-cosymplectic manifolds. These are almost contact metric structures $(\phi, \xi, \eta, g)$ equipped with a closed Lee form $\omega$ and a smooth function $f$ satisfying
$$
d\eta = \omega \wedge \eta, \;\; d\Phi = 2f\eta \wedge \Phi + 2\omega \wedge \Phi,
$$  
where $\Phi(\cdot, \cdot) = g(\cdot, \phi \cdot)$ is the second fundamental form. We derive integrability conditions and prove a dimensional dichotomy: in dimension $3$, $\omega$ may admit transverse components, while in higher dimensions it must be proportional to $\eta$. This rigidity, which contrasts with even-dimensional conformal symplectic geometry, is established and illustrated by explicit examples in dimensions $3$ and $5$. The framework generalizes and unifies prior results on locally conformal almost cosymplectic and almost $f$-cosymplectic structures.
\end{abstract} 
 	

\maketitle
 	
\section{Introduction}
 	
 The study of locally conformal geometric structures begins with a simple yet powerful idea: global geometric conditions may be relaxed to hold only up to local conformal changes. Since Vaisman’s pioneering work on locally conformal symplectic manifolds \cite{Vai}, this perspective has deeply influenced differential geometry, extending to locally conformal K\"ahler, cosymplectic, and contact-type geometries \cite{chin, MadMass2, Mass1, Ol1}. In each case, a closed $1$-form the Lee form governs the structure, capturing the obstruction to global conformality.
 	
 In odd dimensions, almost contact metric geometry offers a natural counterpart to symplectic and Hermitian theories. Here, almost cosymplectic manifolds, defined by the closedness of both the contact form $\eta$ and the second fundamental form $\Phi$, have been thoroughly investigated \cite{golb, MadMass1, Ol2}. A more recent generalization, almost $f$-cosymplectic manifolds \cite{aym}, allows a smooth function $f$ to interpolate between cosymplectic and Kenmotsu geometries. This introduces a useful flexibility, yet a parallel locally conformal theory for such generalized $f$-cosymplectic structures has remained missing.
 	
 In this paper, we fill that gap by introducing and systematically analyzing \emph{locally conformal almost generalized $f$-cosymplectic manifolds}. Our work unifies and extends earlier studies while uncovering a striking rigidity phenomenon that distinguishes odd-dimensional geometry from its even-dimensional analogues. 
 	
 We begin by characterizing these manifolds through a closed Lee form $\omega$ satisfying the compatible differential system
 $$
 d\eta = \omega \wedge \eta, \;\; d\Phi = 2f\eta \wedge \Phi + 2\omega \wedge \Phi.
 $$
 This characterization (Theorem \ref{TheoImport1}) not only generalizes earlier results for almost cosymplectic and almost $f$-cosymplectic manifolds \cite{aym, MadMass1, Mass1, Ol1} but also leads to a natural integrability condition linking $f$ and $\omega$ (Proposition \ref{prop:integrability-f}). What emerges is a clear dimensional dichotomy: in dimension $3$, an example shows that $\omega$ need not align with $\eta$, while in dimensions $5$ and higher, $\omega$ is forced to be proportional to $\eta$, a rigidity absent in the even-dimensional symplectic setting.
 	
 The implications of this rigidity are geometric as well as topological, affecting the Reeb foliation, curvature identities, and the global structure of such manifolds. To ground the theory, we provide explicit examples in dimensions $3$ and $5$, demonstrating both the low-dimensional flexibility and the high-dimensional constraints.
 	
 The paper is organized as follows. Section \ref{Generalized} reviews almost contact metric structures and introduces almost generalized $f$-cosymplectic manifolds, including a basic constraint on $f$ (Proposition \ref{fconstraint}) based on an injectivity lemma (Lemma \ref{lem:injectivity-wedge}). Section \ref{LcGeneralized} develops the locally conformal theory, presenting the main characterization, integrability conditions, examples, and rigidity theorems. Section \ref{Conclusion} summarizes our findings and suggests directions for future research.
 	
 Throughout, all manifolds and tensor fields are smooth, and we follow the conventions of almost contact metric geometry as in \cite{bl}.

 	\section{Generalized $f$-cosymplectic manifolds}\label{Generalized}

 	Let $M$ be a $(2n + 1)$-dimensional manifold equipped with an almost contact structure $(\phi, \xi, \eta)$, that is, $\phi$ is a tensor field of type $(1, 1)$, $\xi$ is a vector field, and $\eta$ is a $1$-form satisfying \cite{bl}
 	\begin{equation}\label{EquaImport1}
 		\phi^{2} = -\mathbb{I} + \eta\otimes\xi,\;\;\eta(\xi)= 1. 
 	\end{equation}
 	This implies that $\phi\xi=0$, $\eta\circ\phi=0$, and $\mathrm{rank}\, \phi = 2n$. 
 	In this case, $(\phi, \xi, \eta,\, g)$ is called an almost contact metric structure on $M$ if $(\phi, \xi, \eta)$ is an almost contact structure on $M$ and $g$ is a Riemannian metric on $M$ such that (see \cite{bl})
 	\begin{equation}\label{EquaImport2}
 		g(\phi\,X, \phi\,Y) = g(X, Y) -  \eta(X)\eta(Y),
 	\end{equation}
 	for any vector field $X$, $Y$ on  $M$. It is easy to see the $(1,1)$-tensor field $\phi$ is skew-symmetric and so 
 	$$
 	\eta(X) =  g(\xi,X).
 	$$
 	The second fundamental form of $M$ is defined by
 	$$ 
 	\Phi(X, Y) = g(X, \phi Y),
 	$$ 
 for any vector fields $X$ and $Y$ on $M$.
 
 The almost contact manifold $(M, \phi, \xi, \eta, g)$ is said to be almost cosymplectic if the forms $\eta$ and $\Phi$ are closed, that is, $d\eta = 0$ and $d\Phi = 0$, $d$ being the operator of the exterior differentiation (see \cite{golb} for more details). If $M$ is almost cosymplectic and its almost contact structure $(\phi,\xi,\eta)$ is normal, then $M$ is called cosymplectic. By normality, we mean that the torsion tensor field $N$ given by 
 $$
N=  [\phi, \phi] + 2d\eta\otimes\xi 
 $$
vanishes, where $[\phi, \phi]$ is the Nijenhuis torsion of  $\phi$ defined by
 $$
 [\phi, \phi](X, Y ) = \phi^{2} [X, Y ] + [\phi X, \phi Y ] - \phi[\phi X, Y ] - \phi[X, \phi Y ],
 $$
 for any vector fields $X$ and $Y$ on $M$.
 
 Now, we adapt the definition of almost $f$-cosymplectic manifolds given \cite{aym} as follows. We say that $M$ is an \textit{almost generalized $f$-cosymplectic manifold} if there exists  an almost contact metric structure $(\phi, \xi, \eta, g)$ on $M$ satisfying  
 \begin{equation}
 d\eta = 0 \;\; \mbox{and} \;\; d\Phi = 2f\eta \wedge \Phi,	
 \end{equation}
 where $f$ is a smooth function on $M$. 
 
This definition extends several well-known classes of almost contact metric manifolds. 
 \begin{itemize}
 	\item[-] When $f \equiv 0$, we recover the classical almost cosymplectic condition $d\eta = 0 = d\Phi$ (see for instance \cite{Ol2} for details).
 	\item[-] When $f$ is a nonzero constant, the structure is known as \emph{almost Kenmotsu}, which appears naturally in warped product constructions and geometries with transverse K\"ahler leaves \cite{dileo}.
 	\item[-] For general $f$, the definition generalizes the ones given by Aktan \textit{et al.} in \cite{aym} in which $f$ satisfies an extra condition $df\wedge \eta=0$, interplays between the Reeb direction and the symplectic part of the structure, and the authors in \cite{kim}. 
 \end{itemize}
Geometrically, the condition $d\eta = 0$ ensures that the $1$-form $\eta$ defines a codimension $1$ foliation on $M$, whose leaves inherit an almost K\"ahler structure via $(\phi, g)$. The last condition, $d\Phi = 2f\eta \wedge \Phi$, implies that the second fundamental form $\Phi$ is not closed, but its exterior derivative is entirely controlled by the function $f$ and the contact form $\eta$. In particular, on each leaf of the foliation $\ker\eta$, the $2$-form $\Phi$ restricts to a symplectic form, but its derivative in the transverse direction is proportional to $f$.
 
 The function $f$ thus serves as a modulating parameter that interpolates between cosymplectic geometry ($f=0$) and Kenmotsu geometry ($f$ constant). In the general case where $f$ is non constant, the structure exhibits a richer local behavior, allowing the geometry to vary from point to point. This flexibility makes almost generalized $f$-cosymplectic manifolds a natural setting for studying geometric structures with non-uniform Lee-type behavior, conformal changes, and foliations with varying transverse geometry.
 
 It is worth noting that when the almost contact structure is normal, an almost generalized $f$-cosymplectic manifold becomes a generalized $f$-cosymplectic manifold in the strict sense. In such cases, the leaves of $\ker\eta$ are K\"ahler manifolds, and the Reeb vector field $\xi$ generates a geodesic flow orthogonal to the leaves.

 A direct consequence of the structure equations is the constraint on the function $f$. Before stating that, we present the following standard Lefschetz-type injectivity result as follows.
 \begin{lemma}\label{lem:injectivity-wedge}
 	Let $(M, \phi, \xi, \eta, g)$ be a $(2n+1)$-dimensional almost contact manifold, and let $\Phi$ be its second fundamental form. Consider the linear map 
 	$$
 	L : \Gamma((\ker\eta)^*) \longrightarrow \Gamma(\Lambda^{2n-1} (\ker\eta)^*), \qquad \alpha \longmapsto \alpha \wedge \Phi^{n-1}.
 	$$
 	Then $L$ is injective, that is, if $\alpha \wedge \Phi^{n-1} = 0$ on $\ker\eta$, then $\alpha = 0$. 
 \end{lemma} 
 \begin{proof}
 	Let $\mathcal{V} = \ker\eta $, a $2n$-dimensional symplectic vector space with symplectic form $\Omega := \Phi_{|_{\mathcal{V}}}$. The problem reduces to showing that the map
 	$$
 	L : \mathcal{V}^{*} \longrightarrow \Lambda^{2n-1} \mathcal{V}^{*}, \;\;\; \alpha \longmapsto \alpha \wedge \Omega^{n-1}
 	$$
 	is injective. Here $\mathcal{V}^{*}$ is the dual of the vector space $\mathcal{V}$. For any $X \in\Gamma(\mathcal{V})$, contraction with the top form $\Omega^{n}$ gives
 	$$
 	i_X(\Omega^{n}) = n \,(i_X \Omega)\wedge \Omega^{n-1}.
 	$$
 	Since $\Omega$ is nondegenerate, the map
 	$$
 	\mathcal{V} \longrightarrow \mathcal{V}^*, \;\; X \longmapsto i_X \Omega,
 	$$
 	is an isomorphism. Also, since $\Omega^{n}\neq 0$, the map
 	$$
 	\mathcal{V} \longrightarrow \Lambda^{2n-1} \mathcal{V}^*, \;\; X \longmapsto i_{X} \Omega^{n}
 	$$
 	is an isomorphism. Combining these, for $\alpha = i_{X} \Omega$ we have
 	$$
 	L(\alpha) = \alpha \wedge \Omega^{n-1} = \frac{1}{n} i_{X} \Omega^{n}.
 	$$
 	Hence $L$ is a composition of isomorphisms, up to a nonzero scalar $1/n$, and in particular injective. Therefore $\alpha \wedge \Omega^{n-1} = 0$ implies that $\alpha = 0$.
 	This can be seen explicitly, as follows. Choose a symplectic basis $\{e_{1},\;\dots,\;e_{n},\;f_{1},\;\dots,\;f_{n}\}$ of $\mathcal{V}$ such that
 	$$
 	\Omega(e_{i}, f_{j}) = \delta_{ij},\;\; \Omega(e_{i}, e_{j}) = \Omega(f_{i}, f_{j}) = 0,
 	$$
 	and let $\{e_{1}^{*},\; \dots,\; e_{n}^{*},\; f_{1}^{*},\; \dots,\; f_{n}^{*}\}$ be the dual basis of $\mathcal{V}^{*}$. Write
 	$ 
 	\alpha = \sum_{i=1}^{n} a_{i} e_{i}^{*} + b_{i} f_{i}^{*}.
 	$ 
 	A straightforward combinatorial computation gives
 	$$
 	e_{i}^{*}\wedge \Omega^{n-1} = (n-1)!\;\nu_{(\widehat{f_{i}})},\;\;
 	f_{i}^{*}\wedge \Omega^{n-1} = (n-1)!\;\nu_{(\widehat{e_{i}})},
 	$$
 	where $\nu:= e_{1}^{*}\wedge f_{1}^{*}\wedge \cdots \wedge e_{n}^{*}\wedge f_{n}^{*}$ is the volume form on $\mathcal{V}$, and $\nu_{(\widehat{e_i})}$ (respectively $\nu_{(\widehat{f_i})}$) is obtained from $\nu$ by omitting $e_i^*$ (respectively $f_i^*$). Thus
 	$$
 	\alpha\wedge\Omega^{n-1} = (n-1)! \sum_{i=1}^{n} \left\{ a_{i} \nu_{(\widehat{f_{i}})} + b_{i} \nu_{(\widehat{e_{i}})}\right\}.
 	$$	
 	Since the $(2n-1)$-forms $\{\nu_{(\widehat{e_i})}, \;\nu_{(\widehat{f_i})}: i=1,\;\dots,\; n\}$ are linearly independent, the equality $\alpha\wedge\Omega^{n-1} = 0$ forces all coefficients $a_i$ and $b_i$ to vanish. Hence $\alpha = 0$. Since $p\in M$ was arbitrary, this argument holds at all points, proving that $L$ is injective.
 \end{proof}
 The following is the constraint conditions on the smooth function $f$. 
 \begin{proposition}\label{fconstraint}
 Let $(M, \phi, \xi, \eta, g)$ be a $(2n+1)$-dimensional almost generalized $f$-cosymplectic manifold. Then the function $f$ satisfies
 $$
 df \wedge \eta \wedge \Phi = 0.
 $$
Furthermore,
 	\begin{enumerate}
 		\item If $n = 1$ (i.e., $\dim M = 3$), the condition imposes no restriction on $f$.
 		\item If $n \geq 2$ (i.e., $\dim M \geq 5$), then $df$ is proportional to $\eta$, that is, there exists a smooth function $\rho$ on $M$ such that
 		$ 
 		df = \rho \eta.
 		$ 
 		In particular, $f$ is constant on the leaves of the foliation $\ker \eta$.
 	\end{enumerate}
 \end{proposition} 
 \begin{proof}
 From the defining conditions $d\eta = 0$ and $d\Phi = 2f \eta \wedge \Phi$, we take the exterior derivative of the latter:
 $$
 d^2\Phi = 2\, df \wedge \eta \wedge \Phi + 2f\, d(\eta \wedge \Phi).
 $$
 Since $d^2\Phi = 0$ and $d\eta = 0$, we compute
\begin{equation}
d(\eta \wedge \Phi)  = d\eta \wedge \Phi - \eta \wedge d\Phi =  - \eta \wedge (2f \eta \wedge \Phi)  = 0.\nonumber
\end{equation} 
Hence,
\begin{equation}\label{TripleDiffer1}
 df \wedge \eta \wedge \Phi = 0.
\end{equation} 
When $n = 1$, the manifold is $3$-dimensional. The $3$-form $\eta \wedge \Phi$ is nowhere zero (in fact, it is a volume form), but equation (\ref{TripleDiffer1}) is a $4$-form on a $3$-manifold, so it holds automatically. Thus no constraint on $f$ arises.
 	
Assume now that $n \geq 2$ (so $\dim M \geq 5$). Decompose $df$ as
$$
df = \rho \eta + \beta,
$$
where $\rho \in C^\infty(M)$ and $\beta$ is a $1$-form such that $\beta(\xi) = 0$, i.e., $\beta$ belongs to the annihilator of $\xi$, which is the cotangent space of the leaves of $\ker \eta$. Substituting into (\ref{TripleDiffer1}) gives
\begin{align}
0 & = (a \eta + \beta) \wedge \eta \wedge \Phi\nonumber\\
& = \beta \wedge \eta \wedge \Phi \nonumber\\
&= - \eta \wedge (\beta \wedge \Phi).\nonumber
\end{align}
Since $\eta$ is nowhere vanishing, it follows that $\beta \wedge \Phi = 0$ when restricted to vectors in $\ker \eta$. Now consider the restriction of $\beta$ to the distribution $\mathcal{V} = \ker \eta$. On each leaf, $\Phi|_\mathcal{V}$ is a nondegenerate $2$-form. By Lemma~\ref{lem:injectivity-wedge}, the map $\alpha \longmapsto \alpha \wedge \Phi$ is injective on sections of $\mathcal{V}^*$. Therefore, $\beta = 0$ on $\mathcal{V}$, and since $\beta(\xi)=0$, we have $\beta = 0$ everywhere. Hence $df = \rho \eta$, so $df$ is proportional to $\eta$.
 
Finally, if $X$ is any vector field tangent to $\ker \eta$, then
$  
X(f) = df(X) = \rho \eta(X) = 0,
$ 
which shows that $f$ is constant on the leaves of the foliation.
 \end{proof} 
In \cite{aym}, the authors study almost $f$-cosymplectic manifolds under the extra hypothesis $df\wedge\eta=0$. Proposition \ref{fconstraint} shows that in dimensions $5$ and higher this condition is automatically satisfied, while in dimension $3$ it is an independent constraint. Consequently, the results of \cite{aym} apply without change in dimensions $\ge 5$, but in dimension $3$ they require the additional assumption $df\wedge\eta=0$. Our definition of almost generalized $f$-cosymplectic manifolds (without that extra condition) is therefore strictly more general in dimension $3$, while in higher dimensions the two definitions coincide.

 \section{Locally conformal almost generalized $f$-cosymplectic manifolds}\label{LcGeneralized}
 	
 Let  $(M,\phi,\xi,\eta,g) $ be an almost contact metric manifold. Such a manifold  is said \textit{locally conformal (l.c.) almost generalized $f$-cosymplectic} if there exists an open covering $\{U_{t}\}_{t \in I}$ of $M$ and smooth functions $\sigma_{t} : U_{t} \longrightarrow \mathbb{R}$ such that, on each $U_{t}$, the almost contact metric structure $ (\phi_{t},\xi_{t},\eta_{t},g_{t})$ 
 $$
 \phi_{t} = \phi ,\;\; \xi_{t} = \exp(\sigma_{t})\xi ,\;\; \eta_{t} = \exp(- \sigma_{t})\eta ,\;\; g_{t} = \exp(- 2\sigma_{t})g, 
 $$ 
 is almost generalized $f\exp(\sigma_{t}) $-cosymplectic.
 
 If the structure $(\phi_{t},\xi_{t},\eta_{t},g_{t})$ is almost generalized $f\exp(\sigma_{t})$-cosymplectic, then we have 
 \begin{equation} \label{fExpCosym1}
 d\eta_{t}  = 0\:\;\mbox{and}\;\; d\Phi_{t} = 2 f \exp(\sigma_{t})\eta_{t}\wedge\Phi_{t}.
 \end{equation} 
 Denote the local $f$-function on $U_{t}$ by
 $$
 f_{t} := f\exp(\sigma_{t}),
 $$ 
 so the requirement in our definition is that, on each $U_{t}$, the structure $(\phi_{t}, \xi_{t}, \eta_{t}, g_{t})$ satisfies the almost generalized $f_{t}$-cosymplectic equations.
 
 On an intersection $U_{t}\cap U_{s}$, we have the direct relations
 \begin{align}
 \xi_{s} & = \exp(\sigma_{s}-\sigma_{t}) \xi_{t},  \;\;
 \eta_{s} = \exp(- (\sigma_{s}-\sigma_{t}))\eta_{t},\nonumber\\ 
 g_{s} & = \exp(- 2(\sigma_{s}-\sigma_{t}))g_{t},\;\;
 \Phi_{s}   = \exp(- 2(\sigma_{s}-\sigma_{t})) \Phi_{t}.\nonumber	
\end{align}	
Hence the $t$- and $s$-structures are conformally equivalent on the  intersection with conformal factor $\exp(- (\sigma_{s}-\sigma_{t}))$. 	

Write the $t$-equation on $U_{t}$, the first relation of (\ref{fExpCosym1}), gives
\begin{align}
0 &= d\eta_{t} = d(\exp(- \sigma_{t})\eta)\nonumber\\
& = -\exp(- \sigma_{t}) d	\sigma_{t}\wedge \eta +\exp(- \sigma_{t})d\eta, \nonumber
\end{align} 
that is, 
\begin{equation}\label{LcdEta1}
d\eta =  d\sigma_{t}\wedge \eta.
\end{equation}
The second relation of (\ref{fExpCosym1}), we have
\begin{align}
2 f \exp(-2\sigma_{t})\eta\wedge\Phi & =  2 f \exp(\sigma_{t})\eta_{t}\wedge\Phi_{t}\nonumber\\
&= d( \exp(-2\sigma_{t})\Phi)\nonumber\\
&= \exp(-2\sigma_{t}) \{ - 2 d\sigma_{t}\wedge  \Phi +  d\Phi  \}.\nonumber
\end{align}
That is, 
\begin{equation} \label{LcdPhi1}
 d\Phi = 2 f_{t} \exp(- \sigma_{t}) \eta\wedge\Phi +  2 d\sigma_{t}\wedge\Phi. 
\end{equation}  
The same calculations for $U_{s}$ gives
\begin{equation}\label{LcdEtaPHi1} 
d\eta =  d\sigma_{s}\wedge \eta\;\;\mbox{and}\;\; d\Phi = 2 f_{s} \exp(- \sigma_{s}) \eta\wedge\Phi +  2 d\sigma_{s}\wedge\Phi. 	
\end{equation}
Subtracting (\ref{LcdEta1}) for $t$ and $s$ yields 
\begin{equation}\label{SigmaST1}
d(\sigma_{t} -\sigma_{s})\wedge\eta  = 0.	
\end{equation}
Similarly, subtracting (\ref{LcdPhi1}) gives
\begin{equation}\label{SigmaST2}
d(\sigma_{t} -\sigma_{s})\wedge\Phi = \{  f_{s} \exp(- \sigma_{s}) - f_{t} \exp(- \sigma_{t})\}\eta\wedge\Phi.	
\end{equation} 
From (\ref{SigmaST1}) there exists a function $h_{ts}$ such that
\begin{equation}
d \sigma_{t} - d \sigma_{s} = h_{ts}\eta,\nonumber	
\end{equation}
which implies that 
\begin{equation}
h_{ts} = \xi(\sigma_{t} - \sigma_{s}).\nonumber
\end{equation}
Substituting into (\ref{SigmaST2}) and using the non-degeneracy of $\Phi$, we obtain one gets the scalar relation
\begin{equation}
\xi(\sigma_{t} - \sigma_{s}) = f_{s} \exp(- \sigma_{s}) - f_{t} \exp(- \sigma_{t}). \nonumber
\end{equation} 
But by construction, the local functions $f_{t}$ define a globally well-defined function $f$ on $M$ via $f_{t} = f\exp(\sigma_{t})$, so
$$
f_{s} \exp(- \sigma_{s})=f_{t} \exp(- \sigma_{t}) = f,
$$
hence 
$$
\xi(\sigma_{t} - \sigma_{s}) = 0.
$$
Therefore, $d \sigma_{t} = d \sigma_{s}$ on $U_{t}\cap U_{s}$, so the local $1$-form $d\sigma_{t}$ glue togther to a globally defined closed $1$-form $\omega$ on $M$ so that $\omega|_{U_{t}} = d \sigma_{t}$. Conversely, assume there exists a globally defined closed $1$-form $\omega$ on $M$ such that  
\begin{equation} \label{DetaDPhi1}
d\eta =  \omega\wedge \eta\;\;\mbox{and}\;\; d\Phi = 2 f  \eta\wedge\Phi +  2 \omega\wedge\Phi.
\end{equation}
By the classical Poincarr\'e  Lemma, there is an open cover $\lbrace U_{t}\rbrace_{t\in I} $ of $M$ and a family  of $C^{\infty}(M)$ functions $ \sigma_{t} : U_{t} \longrightarrow \mathbb{R}$ such that $\omega  = d\sigma_{t}$ on $U_{t}$. By (\ref{DetaDPhi1}), one has 
$$
d\eta =  d\sigma_{t}\wedge \eta\;\;\mbox{and}\;\; d\Phi = 2 f  \eta\wedge\Phi +  2 d\sigma_{t}\wedge\Phi.
$$
Multiplying the first equation by $\exp(- \sigma_{t})$ and second equation by $\exp(- 2\sigma_{t})$, we obtain
$$
d(\exp(- \sigma_{t})\eta)=0\;\;\mbox{and}\;\;d\Phi_{t} = 2 f \exp(\sigma_{t})\eta_{t}\wedge\Phi_{t}.
$$
Hence, escaled the structure $(\phi_{t},\xi_{t},\eta_{t}, g_{t})$ is almost $f\exp(\sigma_{t})$-cosymplectic on $U_{t}$, so $M$ is locally conformal almost generalized $f$-cosymplectic. 

Therefore, we have the following result.
 \begin{theorem} \label{TheoImport1}
 An almost contact metric manifold $M$ is an l.c. almost generalized $f$-cosymplectic if and only if there exists a globally defined closed $1$-form $\omega$ on $M$ such that   
\begin{equation} \label{CharacterRela}
d\eta = \omega\wedge \eta\;\;\mbox{and}\;\;  d\Phi  = 2 f \eta\wedge\Phi  + 2 \omega\wedge  \Phi.
\end{equation} 
 \end{theorem} 
 We first investigate the constraints that may be attached the smooth function $f$ defined in the case of the l.c. almost generalized $f$-cosymplectic structure. 
 \begin{proposition}\label{prop:integrability-f}
 	Let $(M,\phi,\xi,\eta,g)$ be an almost contact metric manifold equipped with a closed $1$-form $\omega$ and a smooth function $f$ satisfying
 	$$
 	d\eta = \omega \wedge \eta \;\;\mbox{and}\;\; d\Phi = 2f\,\eta \wedge \Phi + 2\omega \wedge \Phi.
 	$$
 	Then the following integrability condition holds:
 	\begin{equation}\label{EquaFunctionF1}
 	(df + f\,\omega) \wedge \eta \wedge \Phi = 0. 	 
 	\end{equation} 	 
 	Consequently:
 	\begin{enumerate}
 		\item If $\dim M = 3$ (i.e., $n=1$), then $\eta\wedge\Phi$ is a volume form and (\ref{EquaFunctionF1}) forces
 		$ 
 		df = - f\,\omega.
 		$ 
 		\item If $\dim M \geq 5$ (i.e., $n\geq 2$), then the $1$-form $df+f\,\omega$ is proportional to $\eta$, i.e., there exists a smooth function $\lambda$ such that 
 		$ 
 		df + f\,\omega = \lambda\,\eta.
 		$ 
 	\end{enumerate}
 \end{proposition} 
 \begin{proof}
 	The closure of $\omega$ and the structure equations imply the consistency condition $d^2\Phi = 0$.  Compute: 	
 	\begin{equation}
 		d(d\Phi)  = 2\,df\wedge\eta\wedge\Phi + 2f\,d(\eta\wedge\Phi) 
 		+ 2\,d\omega\wedge\Phi - 2\omega\wedge d\Phi.\nonumber
 	\end{equation} 
 	Using $d\omega=0$ and substituting $d\Phi$ again,
 	\begin{align}
 		0 &= 2\,df\wedge\eta\wedge\Phi + 2f\,d(\eta\wedge\Phi) 
 		- 2\omega\wedge\bigl(2f\,\eta\wedge\Phi + 2\omega\wedge\Phi\bigr) \nonumber\\
 		&= 2\,df\wedge\eta\wedge\Phi + 2f\,d(\eta\wedge\Phi) 
 		- 4f\,\omega\wedge\eta\wedge\Phi - 4\,\omega\wedge\omega\wedge\Phi.\nonumber
 	\end{align}
 	Since $\omega\wedge\omega=0$, we obtain
 	\begin{equation}\label{EquaFunctionF2}
 	0 = df\wedge\eta\wedge\Phi + f\,d(\eta\wedge\Phi) - 2f\,\omega\wedge\eta\wedge\Phi.   
 	\end{equation}
 	Now compute $d(\eta\wedge\Phi)$ using the given equations 
 	\begin{align}
 		d(\eta\wedge\Phi) &= d\eta\wedge\Phi - \eta\wedge d\Phi\nonumber \\ 
 		&= \omega\wedge\eta\wedge\Phi 
 		- 2f\,\eta\wedge\eta\wedge\Phi 
 		- 2\,\eta\wedge\omega\wedge\Phi \nonumber\\
 		&= 3\,\omega\wedge\eta\wedge\Phi. \nonumber
 	\end{align} 
 	Substituting this into (\ref{EquaFunctionF2}) yields
 	$$
 	0 = df\wedge\eta\wedge\Phi + 3f\,\omega\wedge\eta\wedge\Phi 
 	- 2f\,\omega\wedge\eta\wedge\Phi 
 	= df\wedge\eta\wedge\Phi + f\,\omega\wedge\eta\wedge\Phi.
 	$$
 	Hence
 	$$
 	(df + f\,\omega)\wedge\eta\wedge\Phi = 0,
 	$$
 	which is (\ref{EquaFunctionF1}).
 	
 	If $\dim M = 3$, the $3$-form $\eta\wedge\Phi$ is a volume form, so (\ref{EquaFunctionF1}) forces $df+f\omega = 0$, i.e., $df = -f\omega$.
 	
 	If $\dim M \geq 5$, write $df+f\omega = \lambda \eta + \beta$, where $\lambda \in C^\infty(M)$ and $\beta(\xi)=0$.  Then
 	\begin{equation}
 	0  = (\lambda\eta+\beta)\wedge\eta\wedge\Phi  = - \eta\wedge(\beta\wedge\Phi).\nonumber
 \end{equation}
 	Since $\eta$ is nowhere zero, we have $\beta\wedge\Phi = 0$ when restricted to $\ker\eta$.  By Lemma~\ref{lem:injectivity-wedge}, the map $\alpha\longmapsto\alpha\wedge\Phi$ is injective on $(\ker\eta)^*$; therefore $\beta=0$ and $df+f\omega = \lambda\eta$, which completes the proof.
 \end{proof}  
 As an example in dimension $3$, we have the following.
 \begin{example}\label{Example1}
 \emph{Consider the $3$-dimensional manifold 
 $$ 
 M^{3}=\{p\in \mathbb{R}^{3}: x\neq 0,\; z\neq 0\},
 $$ 
 where $p=(x,y,z)$ are the standard coordinates in $\mathbb{R}^3$. Define an almost contact metric structure by
$$
\eta = \exp(x)\,dz, \; \;  \xi =\exp(-x) \,\partial_z,
$$
with the $(1,1)$-tensor $\phi$ and metric $g$ given by
$ 
\phi \partial_{x} =\partial_{y},\; \phi \partial_{y} = -\partial_{x}, \;  \phi \xi = 0
$
and
$
g = \exp(z) (dx^{2} +dy^{2}) + \exp(2x) dz^{2}.
$  
It is straightforward to verify that (\ref{EquaImport1}) and (\ref{EquaImport2}), so $(\phi,\xi,\eta,g)$ is an almost contact metric structure. The non-vanishing components of the fundamental $2$-form $\Phi$, with respect to the basis $\{\partial_x,\partial_y,\partial_z\}$, are given by 
$ 
\Phi(\partial_{x}, \partial_{y})= -  \exp(z) 
$
and
$
\Phi(\partial_{y}, \partial_{x}) =  \exp(z).
$ 
Hence
$ 
 \Phi = -  \exp(z)\, dx \wedge dy.
$ 	
Let us define the closed $1$-form  
$ 
\omega = dx. 
$
It is easy to see that $d\omega=0$. A direct computation shows
$ 
d\eta = d( \exp(x) dz) =  \exp(x)  dx\wedge dz = \omega\wedge\eta.
$  	   
Next, since $\Phi=-  \exp(z) dx\wedge dy$, we have $d\Phi = -  \exp(z) dz \wedge dx\wedge dy$. The structural condition requires
$ 
d\Phi = 2f\,\eta\wedge \Phi + 2\omega \wedge \Phi.
$ 
Now,
\begin{align}
2(f\eta+\omega)\wedge\Phi & = 2(f \exp(x)dz + dx )\wedge (-  \exp(z)\,dx\wedge dy)\nonumber\\
& = - 2 \exp(z+x)f\,dz \wedge dx\wedge dy. \nonumber
\end{align} 
So, we have
$ 
 2 \exp(z+x)f  =   \exp(z)  
$ 
 which leads to
$ 
f(x) = \frac{1}{2} \exp(-x)\neq 0.
$
Finally, we verify the integrability condition $df = -f\omega$:
$$
df = d\left(\frac{1}{2} \exp(-x)\right) = -\frac{1}{2}  \exp(-x) dx = -f\,dx = -f\,\omega.
$$  	
Thus, $f$ is non-vanishing and satisfies $df = -f\omega$ as required. Consider the open neighborhood $U$ of $M$ given by $U:= \{p\in M^{3}: x> 0\}$, and there exists a differentiable function $\sigma$ on $U$ such that $\omega = d\sigma$, where $\sigma = x  + k$ with $k$ a constant. 	Finally,$\omega = dx$ and $\eta$ are not proportional, since $dx$ and $dz$ are linearly independent at each point.
 Therefore, $(M^{3}, \phi, \xi, \eta, g)$ is an l.c. almost generalized $f$-cosymplectic manifold with $f(x) =  \frac{1}{2} \exp(-x)$ and $\omega$ not proportional to $\eta$..
  }	
 \end{example}
 This establishes the existence of such a structure on a $3$-dimensional manifold, and we have the following results. 
 \begin{theorem}
 A $3$-dimensional manifold admits a locally conformal almost generalized $f$-cosymplectic structure for which the Lee form $\omega$ is not proportional to $\eta$.
 \end{theorem} 
The following is a Corollary to Theorem \ref{TheoImport1}.
\begin{corollary}
Let $(M,\phi,\xi,\eta, g)$ be an l.c. almost generalized $f$-cosymplectic manifold.
Then there exists a closed $1$-form $\omega$ such that 
$$ 
d\eta=\omega\wedge\eta\;\;\mbox{and}\;\;d\Phi = 2f \eta\wedge\Phi + 2\omega\wedge\Phi.
$$
\begin{enumerate}
\item[(i)]  If $f=0$ and $\omega=d\sigma$ locally (or globally if $\omega$ is exact), then under the conformal change 
$ 
\phi'=\phi, \; \xi' = \exp( \sigma)\xi, \;  \eta' = \exp(-\sigma) \eta, \;  g' = \exp(-2\sigma)g,
$ 
the rescaled structure satisfies
$ 
d\eta' = 0, \;  d\Phi' = 0.
$ 
Hence, $(M,\phi',\xi',\eta',g')$ is an almost cosymplectic manifold. In particular, the above class recovers locally conformal almost cosymplectic manifolds. 
\item[(ii)] If $f$ is a nonzero constant and $\omega=d\sigma$ locally (or globally), then under the same conformal change one has
$ 
d\eta' = 0,	\;  d\Phi' = 2f \exp(\sigma) \eta'\wedge\Phi'.
$ 
Thus, $(M,\phi',\xi',\eta',g')$ is an almost generalized $ f \exp(\sigma) $-cosymplectic manifold. Consequently, the class recovers locally conformal almost generalized $f$-cosymplectic manifolds. 
\end{enumerate}
\end{corollary}
\begin{proof}
Assume that $\omega=d\sigma$ locally (or globally if exact). Let compute the conformal derivatives of $\eta'$ and $\Phi'$. We have 
\begin{align}
d\eta' & = d(\exp(-\sigma)\eta) = \exp(-\sigma)(-d\sigma\wedge\eta + d\eta) =0,\nonumber	\\
\mbox{and}\;\;d\Phi' &= d(\exp(-2\sigma)\Phi) = \exp(-2\sigma)\{-2\omega\wedge\Phi + 2f\,\eta\wedge\Phi + 2\omega\wedge\Phi\} \nonumber\\ 
&= 2f \exp( \sigma)  (\exp(- \sigma)\eta)\wedge(\exp(-2\sigma)\Phi)\nonumber\\
&= 2f\exp(\sigma)\eta'\wedge \Phi'.\nonumber
\end{align}  
If $f=0$ the last identity reads $d\Phi'=0$, so $(\phi', \xi', \eta', g')$ is almost cosymplectic. If $f\neq0$ is constant, we obtain 
$$
d\Phi'=2f\, \exp(\sigma)\,\eta'\wedge\Phi',
$$
and  the structure $(\phi', \xi', \eta', g')$ is almost generalized $f \exp(\sigma)$-cosymplectic. 
\end{proof}  
Note that the terminology \textit{locally conformal} reflects the fact that the closed $1$-form $\omega$ need not be globally exact. If $\omega=d\sigma$ globally, then the manifold is \textit{globally conformal} to an almost generalized $f$-cosymplectic structure. For instance almost cosymplectic if $f=0$, almost Kenmotsu if $f\neq 0$ constant. If $\omega$ is only locally exact, i.e., $\omega|_{U_{t}} = d\sigma_{t}$ on each open set of a covering $\{U_{t}\}$, then the rescaling $\eta_{t} =\exp(-\sigma_{t})\eta$, $\Phi_{t}=\exp(-2\sigma_{t})\Phi$ gives the desired generalized $f$-cosymplectic structures on each $U_t$, but no global conformal factor exists. This is entirely analogous to the distinction between \textit{locally conformal} and \textit{globally conformal} structures in Hermitian and cosymplectic geometry. 
\begin{example}\label{Example2}
 \emph{Let $(N^{2n}, \omega_{N}, g_N)$ be a symplectic manifold of even dimension $2n$, with compatible almost complex structure $J$. Consider the product
$$
M = S^1 \times N,
$$
with angular coordinate $\theta$ on $S^1$.  Then $M$ has odd dimension $(2n+1)$. Define an almost contact metric structure on $M$ by
$$
\xi =  \frac{\partial}{\partial \theta}, \;\;\eta = d\theta, \;\;\phi_{|_{TN}} = J, \;\; \phi \xi = 0,\;\;
g = d\theta^{2} + g_{N}.
$$
The second fundamental form of this structure is
$$ 
\Phi(X,Y) = g(X, \phi Y) = \omega_{N}(X,Y).
$$ 
Since $d\eta = 0$ and $d\Phi = 0$, the manifold $(M,\phi,\xi,\eta,g)$ is cosymplectic.  Now, applying the conformal change
$$
\xi' = e^{ \theta}\xi, \;\;\eta' = e^{-\theta}\eta, \; \; g' = \exp(-2\theta) g, \;\;\Phi' = \exp(-2\theta)\Phi .
$$
A direct computation gives
$$
d\eta' = - d\theta \wedge \eta', \;  d\Phi' = - 2 d\theta \wedge \Phi',
$$
with $f=0$. Therefore $(M, \phi, \xi, \eta', g')$ is a locally conformal cosymplectic manifold with Lee form $\omega = - d\theta$. Since $\omega$ is not globally exact on $S^1$, the structure is not globally conformal cosymplectic.}
\end{example} 
This example illustrates the distinction between locally conformal and globally conformal cosymplectic structures. Therefore we have the following result.
 \begin{theorem}\label{Lefschetzype}
 Let $(M,\phi,\xi,\eta,g)$ be a $(2n+1)$-dimensional almost contact metric manifold with second fundamental $2$-form $\Phi$. Assume that
 $$
 d\eta = \omega \wedge \eta \;\;\mbox{and}\;\; d\Phi = 2f\,\eta\wedge\Phi + 2\omega\wedge\Phi,
 $$
 for some smooth function $f$ and a closed $1$-form $\omega$. If $n\geq 2$ (i.e.\ $\dim M \geq 5$), then $\omega$ must be proportional to $\eta$ (hence $\omega_{|_{\ker\eta}} = 0$).  
 \end{theorem}  
 \begin{proof}
Fix a point $p\in M$ and set $\mathcal{V}:=\ker\eta_p$, a $2n$-dimensional vector space. Let $\Omega =\Phi|_{\mathcal V}$. Since $(M,\phi, \xi, \eta, g)$ is an almost contact metric manifold, $\Omega$ is non-degenerate and hence $\Omega^{n}\in\Lambda^{2n}\mathcal V^{*}$ is a nonzero top-degree form.
  	
From the hypothesis
 $$
 d\Phi=2f \eta\wedge\Phi + 2 \omega\wedge\Phi,
 $$
 we obtain
$$
d \Phi^{n} = n d\Phi\wedge\Phi^{n-1} = 2n f \eta\wedge\Phi^{n} + 2n \omega\wedge\Phi^{n}.
$$
At $p$, we decompose $\omega$ as $\omega=\alpha+h\eta$, where $\alpha(\xi)=0$ and $h=\omega(\xi)$. Then
 $$
  \omega\wedge\Phi^{n} = \alpha\wedge\Phi^{n} + h \eta\wedge\Phi^{n},
  $$
  and hence
  $$
  d \Phi^{n} = 2n(f+h) \eta\wedge\Phi^{n}+2n\,\alpha\wedge\Phi^{n}.
  $$
 Now, we evaluate this identity on vectors $X_1,\dots,X_{2n+1}\in\mathcal{V}$.  Since $\eta(X_i)=0$ for all $i$, we have
$$
(\eta\wedge\Phi^{n})(X_1,\dots,X_{2n+1})=0.
$$
Since $\Phi^{n}$ restricts to a top-degree form on $\mathcal{V}$, its exterior derivative vanishes when evaluated on vectors tangent to $\mathcal{V}$, and   therefore one gets
$$
0=d \Phi^{n}(X_1,\dots,X_{2n+1}) = 2n (\alpha\wedge\Phi^{n})(X_1,\dots,X_{2n+1}).
$$
Thus
$$
\alpha\wedge\Phi^{n}=0 \;\; \text{on} \;\;\mathcal{V}.
$$
Since $\Phi|_{\mathcal V}=\Omega$ is non-degenerate, we can write
$$
\alpha\wedge\Phi^{n}=(\alpha\wedge\Phi)\wedge\Phi^{\,n-1}.
$$
For $n\ge2$, the map
$$
\beta\longmapsto \beta\wedge\Phi^{\,n-1}
$$
is injective on $\mathcal V^*$ by Lemma \ref{lem:injectivity-wedge}. Hence $\alpha\wedge\Phi=0$ on $\mathcal{V}$ implies $\alpha=0$ at $p$. Since $p$ is an arbitrary point, $\alpha$ vanishes identically on $M$, and therefore we have
$$
\omega = h \eta \;\; \text{on}\;\; M.
$$
In particular, $\omega|_{\mathcal{V}}=0$, which completes the proof.
\end{proof}
We therefore deduce the following result. 
 \begin{theorem} \label{LeeProportionalTheorem}
 Let $(M^{2n+1},\phi,\xi,\eta,g)$ be a locally conformal almost generalized $f$-cosymplectic manifold of dimension $\ge 5$. Then, the Lee form $\omega$ is proportional to $\eta$, that is, $\omega = \omega(\xi)\eta$.
 \end{theorem}  
In contrast with the even-dimensional locally conformal symplectic case, where the Lee form may be genuinely transverse \cite{Vai}, odd-dimensional almost contact analogues exhibit strong rigidity. In particular, in dimension at least $5$, Lee-type forms necessarily align with the Reeb direction, as already observed in related contexts for locally conformal almost cosymplectic manifolds \cite{Ol1}.
 
Next, we construct an example of an almost contact metric manifold of dimension $5$ which is not almost generalized $f$-cosymplectic, but is locally conformal almost generalized $f$-cosymplectic and satisfies the conclusion of Theorem~\ref{LeeProportionalTheorem}. 
 \begin{example} \label{Example3}
 \emph{Let
 $ 
 M^{5} =\left\{(x,y_{1},y_{2},z_{1},z_{2})\in\mathbb{R}^{5} :  z_{1}\neq 0\right\},
$ 
with standard coordinates $(x,y_{1},y_{2},z_{1},z_{2})$. Define a $1$-form $\eta$ and a vector field $\xi$ by
$$
\eta =  dz_{1}, \;\; \xi =  \partial_{z_{1}}.
$$   	
Define the $(1,1)$-tensor field $\phi$ by
$$
\phi(\partial_x)=\partial_{y_1},\;\; \phi(\partial_{y_1})=-\partial_x,\;\; 
 \phi(\partial_{z_2})=\partial_{y_2},\;\;\phi(\partial_{y_2})=-\partial_{z_2},\;\;\phi\xi=0.
$$   	
Let the Riemannian metric $g$ be given by
$$
g = \exp(z_1)\bigl(dx^2+dy_1^2+dz_2^2+dy_2^2\bigr) +  dz_1^2.
$$
A direct computation shows that (\ref{EquaImport1}) and (\ref{EquaImport2}) are satisfied. Hence $(\phi,\xi,\eta,g)$ is an almost contact metric structure. The fundamental $2$-form $\Phi(\cdot, \cdot) = g(\cdot,\phi \cdot)$ is given by
$$  
\Phi = - \exp(z_1) \bigl(dx\wedge dy_1 + dz_2\wedge dy_2\bigr).
$$   	
Let $b$ be a smooth non-constant integrable function $b= b(z_{1})$ and define the $1$-form
$$ 
\omega =  b\,dz_1.
$$ 
Clearly $d\omega=0$ and
$ 
 d\eta 	= 	d( dz_1) = 0 = 	\omega\wedge\eta.
$    	
Moreover,
$ 
d\Phi = - \exp(z_1)dz_1\wedge \left(dx\wedge dy_1 + dz_2\wedge dy_2\right).
$ 
On the other hand,  one has 
\begin{equation}
2(f\eta+\omega)\wedge\Phi  = -2\exp(z_1)(f + b)\, dz_1\wedge(dx\wedge dy_1 + dz_2\wedge dy_2).\nonumber
\end{equation}  
Thus the structural equation
$$ 
d\Phi = 2f\,\eta\wedge\Phi + 2\omega\wedge\Phi
$$ 
holds if and only if
$$
2(f + b)=1,
$$
i.e.,
$ 
f = \frac{1}{2}- b.
$  
The integrability condition for $n\ge 2$ gives
$$
df + f\omega = \left\{-b'  + b\left(\frac{1}{2}- b\right) \right\} dz_1 = \lambda \eta
$$
with $\lambda =  - b'  + b\left(\frac{1}{2}- b\right)$.
Let
$ 
U=\{p\in M^{5}: z_1>0\}.
$ 
On $U$ there exists a smooth function
$ 
\sigma = \int b\; dz_1
$ 
such that $\omega = d\sigma$. Therefore, after the conformal change
$$ 
\phi'=\phi,\;\; \xi'=\exp(\sigma)\xi,\;\;\eta'=\exp(-\sigma)\eta,\;\;g'=\exp(-2\sigma) g,
$$ 
the new structure satisfies $d\eta'=0$ and 
$$
d\Phi' = 2f'\eta'\wedge\Phi'\;\;\mbox{with}\;\; f' = \left(\frac{1}{2}- b\right)\exp(\int b\; dz_1).
$$
Hence $(\phi',\,\xi', \, \eta',\, g')$ is an almost generalized $f\exp(\sigma)$-cosymplectic manifold. Since $d\Phi\neq 2f\eta\wedge \Phi$, the original structure is not almost generalized $f$-cosymplectic. Furthermore, the Lee form satisfies
$ 
\omega = \omega(\xi)\,\eta,
$ 
$  	
\omega(\xi)= b,
$ 
which is consistent with Theorem~\ref{LeeProportionalTheorem}.
}
\end{example}

\section{Conclusion}\label{Conclusion} 

This paper introduces and systematically investigates the class of locally conformal almost generalized $ f$-cosymplectic manifolds. Our work bridges a notable gap in the literature by extending the locally conformal framework from almost cosymplectic to generalized $ f$-cosymplectic geometry.

The main contributions of this paper can be summarized as follows: 
\begin{enumerate}
	\item[(1)] We provide a complete description via a closed Lee form $ \omega$ satisfying the structure equations (3.14), generalizing earlier results in almost cosymplectic and almost $f$-cosymplectic settings (Theorem~\ref{TheoImport1}).
	\item[(2)] In dimensions $5$ and higher, $\omega$  must be proportional to $ \eta$ (Theorems \ref{Lefschetzype} and \ref{LeeProportionalTheorem}). This reveals a fundamental structural constraint distinguishing odd-dimensional locally conformal $f$-cosymplectic geometry from its even-dimensional symplectic counterpart \cite{Vai}.
	\item[(3)]  Examples in dimensions $3$ (Example~\ref{Example1}) and $5$ (Example~\ref{Example3}) demonstrate both the flexibility in low dimensions and the necessity of the rigidity in higher dimensions, confirming the non-triviality of the theory.
\end{enumerate}
The established framework opens several natural directions for future research, including:
\begin{itemize}
	\item  Curvature analysis under the locally conformal $f$-cosymplectic condition. Of particular interest would be the classification of Einstein, $\eta$-Einstein, or constant $\phi$-sectional curvature examples. 	 
	\item  Investigating the topological implications of the Lee-form proportionality in compact settings. For instance, does this rigidity impose restrictions on the Betti numbers, the fundamental group, or the existence of certain characteristic classes?	
	\item  Study of the canonical foliations induced by the Reeb vector field and the Lee form, especially in relation to the function $f$. The interplay between the contact structure, the conformal factor, and the function $f$ may lead to new insights into the geometry of foliations. 
	\item  Characterizing when a locally conformal structure is globally conformal to an almost generalized $f$-cosymplectic manifold. This relates to the cohomology class of the Lee form and may involve conditions on the first Betti number or the existence of global parallel forms.
	\item Extending the framework to other almost contact metric families, such as almost Kenmotsu, almost Sasakian, or paracontact geometries. The function $f$ could also be allowed to depend on additional geometric data, leading to richer interpolation phenomena. 
\end{itemize}
We note that several of these research directions are currently being studied by the authors in a forthcoming paper.

 The locally conformal almost generalized $f$-cosymplectic condition provides a rich geometric setting that generalizes several structures of interest in mathematical physics and foliation theory. For instance, the modulating function $f$ may be interpreted in physical contexts as a variable dissipation coefficient, while the rigidity $\omega =\omega(\xi)\eta$ in higher dimensions imposes strong constraints on the geometry of Reeb flows and their transverse symplectic structures. This framework thus opens new avenues in geometric mechanics, thermodynamics, and the study of foliations with compatible contact‑symplectic data.
 
 Overall, this paper lays a solid foundation for the study of locally conformal almost generalized $f$-cosymplectic geometry, a versatile framework that encapsulates several important geometries while exhibiting new and distinctive rigidity properties. We hope that our results will stimulate further research in differential geometry and related fields where contact and symplectic structures play a fundamental role


\begin{thebibliography}{10} 		
 	\bibitem{aym} N. Aktan, M. Yildirim and C. Murathan, Almost $f$-cosymplectic manifolds, Mediterr. J. Math., 11 (2014), no. 2, 775-787.  
 	\bibitem{bl} D. E. Blair, Riemannian geometry of Contact and symplectic manifolds. Progress in Mathematics, vol. 203, Birkh\"auser Boston, Inc., Boston, MA., 2002.
 	\bibitem{chin} D. Chinea and J. C. Marrero, Conformal changes of almost cosymplectic manifolds, Demonstratio Math. 25 (1992), no. 3, 641-656. 
 	\bibitem{dileo} G. Dileo, On the geometry of almost contact metric manifolds of Kenmotsu type, Differential Geom. Appl. 29 (2011), suppl. 1, S58-S64.
 	\bibitem{GS} A. Ghosh, R. Sharma and J. T.  Cho, Contact metric manifolds with parallel torsion tensor, Ann. Global Anal. Geom. 34 (2008), no. 3, 287-299.   	
 	\bibitem{golb} S. I. Golberg and K. Yano, Integrability of  almost cosymplectic structures, Pacific J. Math.31 (1969), 373 -382. 
 	\bibitem{kim} T. W. Kim and H. K. Pak,  Canonical foliations of certain classes of almost contact metric structures, Acta Math. Sin. (Engl. Ser.), 21 (2005), no. 4, 841-846.  
 	\bibitem{MadMass1} S. H.  Maduna and F. Massamba, Certain class of almost cosymplectic manifolds with K\"ahlerian leaves, Mediterr. J. Math. 20 (2023), no. 3, Paper No. 163, 21 pp.  
 	\bibitem{MadMass2} S. H.  Maduna and F. Massamba,  Locally conformal almost cosymplectic manifolds and nullity distributions, Afr. Mat. 36 (2025), no. 4, Paper No. 169, 17 pp.
 	\bibitem{Mass1} F. Massamba and A. Maloko Mavambou,  A class of locally conformal almost cosymplectic manifolds, Bull. Malays. Math. Sci. Soc., 41 (2018), no. 2, 545-563. 
 	\bibitem{Ol1} Z. Olszak,  Locally conformal almost cosymplectic manifolds, Colloq. Math., 57(1989), no. 1, 73-87.  
 	\bibitem{Ol2} Z. Olszak, On almost cosymplectic manifolds, Kodai Math. J., 4 (1981), no. 2, 239-250. 
 	\bibitem{Vai} I. Vaisman, Locally conformal symplectic manifolds, Internat. J. Math. Math. Sci. 8 (1985), no. 3, 521-536.
 	\end{thebibliography}
 \end{document}